\documentclass[11pt]{article}

\usepackage[utf8]{inputenc}
\usepackage[british]{babel}
\usepackage[a4paper, margin=2.5cm]{geometry}
\usepackage{amsmath, amsthm, amssymb, amsfonts}
\usepackage{mathtools}
\usepackage{thmtools}
\usepackage[shortlabels]{enumitem}
\usepackage[font={small, it}]{caption}
\usepackage{microtype}
\usepackage{xcolor}
\usepackage{mathabx}
\usepackage{hyperref}
\usepackage{bm}
\usepackage[normalem]{ulem}

\newcommand{\cA}{\ensuremath{\mathcal A}}

\newcommand{\cC}{\ensuremath{\mathcal C}}

\newcommand{\cE}{\ensuremath{\mathcal E}}

\newcommand{\cG}{\ensuremath{\mathcal G}}

\newcommand{\cP}{\ensuremath{\mathcal P}}

\newcommand{\cS}{\ensuremath{\mathcal S}}

\newcommand{\eps}{\varepsilon}
\renewcommand{\phi}{\varphi}
\renewcommand{\rho}{\varrho}

\let\setminus=\smallsetminus
\let\emptyset=\varnothing

\newcommand{\Gnp}{G_{n, p}}
\newcommand{\Gnm}{G_{n, m}}
\newcommand{\Hnm}{G_{n, m}^H}
\newcommand{\Hpnm}{G_{n, m}^{H'}}
\newcommand{\Hpnmm}{G_{n, m'}^{H'}}

\newcommand{\Hppnmm}{G_{n, m'}^{H''}}

\newcommand\redsout{\bgroup\markoverwith{\textcolor{red}{\rule[0.5ex]{2pt}{0.5pt}}}\ULon}

\declaretheorem[parent=section]{theorem}
\declaretheorem[sibling=theorem]{lemma}

\declaretheorem[sibling=theorem]{claim}

\declaretheorem[sibling=theorem,style=definition]{definition}

\setlist{itemsep=0.1em, topsep=0.1em, parsep=0.1em, partopsep=0.1em}

\hypersetup{
  colorlinks,
  linkcolor={red!60!black},
  citecolor={green!50!black},
  urlcolor={blue!80!black}
}

\colorlet{RoyalRed}{red!70!black}
\definecolor{RoyalBlue}{rgb}{0.25, 0.41, 0.88}
\definecolor{RoyalAzure}{rgb}{0.0, 0.22, 0.66}

\newlength{\bibitemsep}
\newlength{\bibparskip}
\let\oldthebibliography\thebibliography
\renewcommand\thebibliography[1]{%
  \oldthebibliography{#1}%
  \setlength{\parskip}{\bibitemsep}%
  \setlength{\itemsep}{\bibparskip}%
}

\title{A new proof of the K\L R conjecture}
\author{
  Rajko Nenadov\thanks{Google Z\"urich. Email: \texttt{rajkon@gmail.com}.}
}
\date{}

\begin{document}
\maketitle

\begin{abstract}
Estimating the probability that the Erd\H{o}s-R\'enyi random graph $\Gnm$ is $H$-free, for a fixed graph $H$, is one of the fundamental problems in random graph theory. If $m$ is such that each edge of $\Gnm$ belongs to a copy of $H'$ for every $H' \subseteq H$, in expectation, then it is known that $\Gnm$ is $H$-free with probability $\exp(- \Theta(m))$. The K\L R conjecture, slightly rephrased, states that if we further condition on uniform edge distribution, the archetypal property of random graphs, the probability of being $H$-free becomes superexponentially small in the number of edges. While being interesting on its own, the conjecture has received significant attention due to its connection with the sparse regularity lemma, and the many results in random graphs that follow. It was proven by Balogh, Morris, and Samotij and, independently, by Saxton and Thomason, as one of the first applications of the hypergraph containers method. We give a new direct proof using induction.
\end{abstract}

\section{Introduction} 
\label{sec:introduction}

The conjecture of Kohayakawa, \L uczak, and R\"odl \cite{kohayakawa1997k}, \emph{K\L R} for short, was originally posed as an approach to establish the embedding counterpart of the sparse regularity lemma in random graphs. The conjecture, if true, allows one to systematically prove many extremal results in random graphs using the regularity method in much the same way as it is done in the usual dense setting (see \cite{gerke_steger_2005}). Janson, \L uczak and Ruci\'nski \cite[p.232]{janson2011random} called this conjecture ``one of the most important open questions in the theory of random graphs''. Let us briefly mention some of its corollaries:
\begin{itemize}
    \item Tur\'an's problem in random graphs: Many partial results \cite{frankl1986large, furedi1994random, gerke_habil, gerke2004k5,haxell1995turan,haxell1996turan, kohayakawa2004turan,szabo2003turan} before it was finally resolved in breakthroughs by Conlon and Gowers \cite{conlon2016combinatorial} and, independently, Schacht \cite{schacht2016extremal} (without resorting to the K\L R conjecture which was not proven at the time);
    \item Ramsey properties of random graphs: A breakthrough result by R\"odl and Ruci\'nski \cite{rodl1995threshold} and further extensions to the asymmetric version of the problem \cite{kohayakawa1997threshold,kohayakawa2014upper};
    \item A number of other probabilistic analogues of classical results from graph theory, such as the graph removal lemma (which, in fact, requires a stronger form of the K\L R conjecture), are discussed in \cite{conlon2014klr};
    \item A result of \L uczak \cite{luczak2000Hexact} which shows that the probability of $\Gnm$ being $H$-free essentially coincides with the probability that it is $(\chi(H)-1)$-partite, provided $m$ is sufficiently large. As the latter is simple to compute, this gives a very precise estimate on the probability of $\Gnm$ being $H$-free.
\end{itemize}
Some of these results can be (and were) proven by different means, and the beauty of the K\L R conjecture is that it gives a unifying way to prove all of them. Today, many of these results can also be proven using the hypergraph containers method \cite{balogh2015independent,saxton2015hypergraph}, but there are still a few examples, such as a result on anti-Ramsey properties \cite{kohayakawa2018anti_ramsey}, where one benefits from the flexibility of the regularity method. The K\L R conjecture itself also spawned a number of results \cite{gerke2007small,gerke2007probabilistic,gerke2007k, gerke2004k5, kohayakawa1996arithmetic} until it was proven in a rather surprising way using the aforementioned hypergraph containers developed, independently, by Balogh, Morris, and Samotij \cite{balogh2015independent} and Saxton and Thomason \cite{saxton2015hypergraph}.

It should be noted that Conlon, Gowers, Schacht, and Samotij \cite{conlon2014klr} have directly established an embedding lemma which complements the sparse regularity lemma in random graphs, sufficient for most applications. However, their result is, in some way, strictly weaker than the K\L R conjecture and does not apply in the case of asymmetric Ramsey problem \cite{kohayakawa1997threshold} or the result of \L uczak \cite{luczak2000Hexact}.

Rather than diving into the regularity method, we will motivate the conjecture from the point of view of `making random graphs even more random'. Some preparation, however, is nonetheless necessary.

Let $H$ be a graph on the vertex set $\{1, \ldots, k\}$, and denote with $K_n^H$ the graph obtained by replacing each vertex $i$ with a distinct set $V_i$ of size $n$, and each edge $ij \in H$ with a complete bipartite graph between $V_i$ and $V_j$. We say that a copy of $H$ in a graph $G \subseteq K_n^H$ is \emph{canonical} if each vertex $i$ belongs to $V_i$, and $G$ is \emph{$H$-free} if it does not contain a canonical copy of $H$. Given $n, m \in \mathbb{N}$, let $\cG(H, n, m)$ be the family of all subgraphs of $K_n^H$ with exactly $m$ edges between each $V_i$ and $V_j$ for $ij \in H$, and let $\Hnm$ be the uniform probability distribution over $\cG(H, n, m)$.  The model $\Hnm$ is a natural \emph{$H$-partite} analogue of the usual $\Gnm$ model, where we choose $m$ edges from $K_n$ uniformly at random. It should be noted that considering $\Hnm$ instead of $\Gnm$ comes from the aforementioned connection with the sparse regularity lemma, in which $H$-partite configurations naturally arise.

It follows from \cite{janson1990,steger1996countingHfree} that $\Hnm$\footnote{The result in \cite{janson1990} is stated for the $\Gnp$ model, and \cite{steger1996countingHfree} transfers it to $\Gnm$. It is routine to do the same for the $\Hnm$ model.} is $H$-free with probability at most $e^{-\Omega(\Phi(H, m))}$, where
$$
    \Phi(H, m) = \min\{n^{v(H')} (m/n^2)^{e(H')} \colon H' \subseteq H\}
$$
is the expected number of canonical copies of the least frequent subgraph $H' \subseteq H$. In particular, for $m \ge n^{2 - 1/m_2(H)}$, where
$$
    m_2(H) = \max \left\{ \frac{e(H') - 1}{v(H') - 2} \colon H' \subseteq H, v(H') > 2 \right\}
$$
is the so-called \emph{2-density}, we have $\Phi(H, m) = \Theta(m)$. On the other hand, it is easy to see that $\Hnm$ is $H$-free with probability at least, say $(2e)^{-2m}$, as long as $H$ contains a vertex of degree at least $2$. Consider, for example, the case when $H$ is a triangle. Split $V_1$ into two equals parts, $V_1 = U_2 \cup U_3$. The probability that all the edges of $\Hnm$ between $V_1$ and $V_i$ have one endpoint in $U_i$ (for $i \in \{2,3\}$) is at least $(2e)^{-2m}$, and such a graph is clearly triangle-free. The K\L R conjecture states that examples with such atypical edge distribution are the main culprit for $\Hnm$ being $H$-free.

To state the result precisely, let us capture the main property which the previously mentioned example fails to satisfy.

\begin{definition}
Given a bipartite graph $G$ on the vertex set $A \cup B$, where $|A| = |B| = n$, and parameters $\eps, \lambda \in (0, 1]$, we say that $G$ is \emph{$(\eps, \lambda)$-lower-regular} if for every $A' \subseteq A$ and $B' \subseteq B$, satisfying $|A'|, |B'| \ge \eps n$, we have
$$
    d_G(A', B') \ge \lambda d_G(A, B).
$$
Here, $d_G(A', B') = e_G(A', B') / |A'||B'|$ denotes the density of the subgraph induced by $A'$ and $B'$ (if $G$ is clear from the context we omit it from the subscript). 
\end{definition}
Simply put, we require that the density of every sufficiently large induced subgraph is at least a fraction of the density of the whole graph -- a property clearly not satisfied by the given example. We say that a subgraph $G \subseteq K_n^H$ is \emph{$(\eps, \lambda)$-lower-regular} if each induced subgraph $G(V_i, V_j)$ is $(\eps, \lambda)$-regular for $ij \in H$. We can now state the conjecture.

\begin{theorem}[K\L R conjecture -- a weak counting version]
\label{thm:KLR}
For every $H$ and $\lambda, \beta > 0$ there exist $\gamma, \eps, C > 0$ such that if $n$ is sufficiently large and $m \ge C n^{2-1/m_2(H)}$ then
$$
    \Pr[Z_H \le \gamma n^{v(H)} (m/n^2)^{e(H)} \mid \Hnm \text{ is } (\eps, \lambda)\text{-lower-regular}] \le \beta^m,
$$
where $Z_H$ denotes the number of canonical copies of $H$ in $\Hnm$.
\end{theorem}

As $\Hnm$ is $(\eps, \lambda)$-lower-regular with high probability for any given constants $\eps, \lambda \in (0, 1)$, and the uniform edge distribution is the most prominent feature of random graphs, Theorem \ref{thm:KLR} tells us that conditioning on $\Hnm$ being `truly random' drastically decreases the probability of being $H$-free. It is worth pointing out that the property of being $(\eps, \lambda)$-lower-regular deterministically implies the existence of many copies of $H$ if $m \ge d n^2$, for any given constant $d > 0$ (with $\eps$ additionally depending on such $d$). This is, however, not the case if $m = o(n^2)$ (see \cite[Example 3.10]{gerke_steger_2005}), thus Theorem \ref{thm:KLR} does not vacuously hold.

The stated version is stronger than the original one in two ways: (i) we need a fraction of the density instead of $d(V_i', V_j') = (1 \pm \eps) m/n^2$; (ii) we obtain the probability of having significantly less than the expected number of canonical copies instead of just being $H$-free. The version of the K\L R conjecture as stated here was proven by Saxton and Thomason \cite{saxton2015hypergraph}. The original statement of the K\L R conjecture was proven by Balogh, Morris, and Samotij \cite{balogh2015independent}\footnote{The results in \cite{balogh2015independent,saxton2015hypergraph} are stated with $\lambda = 1 - \eps$, but it is easy to adapt them to an arbitrary $\lambda > 0$.}. Both proofs were obtained as corollaries of the powerful hypergraph container method.

The purpose of this paper is to give a new, intuitive proof of Theorem \ref{thm:KLR}. As a consequence, this gives a transparent proof for the bulk of (extremal) results in random graphs.

\section{Preliminaries}
\label{sec:preliminaries}

Given a graph $G$, we denote with $|G|$ the number of edges of $G$. Similarly, given graphs $G$ and $G'$ on the same vertex set, we let $G \setminus G'$ denote the graph obtained from $G$ by deleting all the edges present in $G'$. 

We shall need a more flexible notion of lower-regularity which allows for different bounds depending on the size of sets. Given a function $\delta \colon (0,1] \rightarrow (0,1)$, we say that a bipartite graph $G$ on the vertex set $A \cup B$, where $|A| = |B| = n$, is \emph{$(\eps, \delta)$-lower-regular} if for every $A' \subseteq A$ and $B' \subseteq B$, satisfying $|A'|, |B'| \ge \eps n$, we have
\begin{equation}
    d(A', B') > \delta(\alpha) d(A, B),
\end{equation}
where $\alpha = \min\{|A'|, |B'|\} / n$. The benefit of using a function instead of a constant will be discussed  after Lemma \ref{lemma:restricted_subgraph}.

\subsection{Avoiding (locally) sparse graphs}

The next lemma tells us that if a graph $Q \subseteq K_{n,n}$ is either sufficiently sparse or not lower-regular, then randomly taken $m$ edges from $K_{n,n}$ are likely to contain many edges outside of such $Q$. Both the statement and the proof are inspired by \cite[Lemma 9.2]{balogh2015independent}.

\begin{lemma} \label{lemma:restricted_subgraph}
    Given $\lambda, \eps, \beta > 0$, the following holds for sufficiently large $n$ and $m > n$: Let $Q \subseteq K_{n,n}$ be a bipartite graph such that either $|Q| < \delta(1) n^2$ or $Q$ is not $(\eps, \delta)$-lower-regular, where $\delta \colon (0, 1] \rightarrow (0, 1)$ is any function such that
    $$
        \delta(x) \le \frac{\lambda}{4e} \left( \frac{\beta}{2} \right)^{1/(\lambda x^2)}.
    $$
    Then all but at most $\beta^m \binom{n^2}{m}$ $(\eps, \lambda)$-lower-regular graphs $G \subseteq K_{n,n}$ with $m > n$ edges satisfy $|G \setminus Q| \ge \lambda \eps^2 m / 2$.
\end{lemma}
\begin{proof}
We make use of the standard estimate $\binom{a}{b} \le \left( \frac{e a}{b} \right)^b$ and the following inequality which holds for any $0 \le x \le m < n^2/2$:
\begin{equation} \label{eq:binom_x}
    \binom{n^2}{m - x} \le \left( \frac{2m}{n^2} \right)^{x} \binom{n^2}{m}.
\end{equation}
Inequality \eqref{eq:binom_x} can be derived from the definition of the binomial coefficient followed by crude estimates.

Observe that, no matter which case we are in, there exist disjoint sets $A, B \subseteq V(K_{n,n})$ of size $|A|, |B| \ge \eps n$ such that $|Q(A, B)| \le \delta(\alpha)|A||B|$, where $\alpha = \min\{|A|, |B|\} / n$. As both $A$ and $B$ are sufficiently large, an $(\eps, \lambda)$-lower-regular graph $G$ with $m$ edges has to contain at least $\ell = \lambda |A||B| m/n^2 \ge \lambda \alpha^2 m$ edges between $A$ and $B$. If $\ell > 2\delta(\alpha)|A||B|$, which happens for $m \ge n^2/2$ (with room to spare), the lemma deterministically holds. Otherwise, denote with $x$ the number of edges in $G(A, B)$ which belong to $Q$. We get the following bound on the number of graphs $G$ with $x > \ell / 2$:
\begin{equation} \label{eq:P_not_regular_calc}
    \sum_{x > \ell/2} \binom{ \delta(\alpha) |A||B|}{x} \binom{n^2}{m - x} < m \binom{\delta(\alpha)|A||B|}{\ell/2} \binom{n^2}{m-\ell/2} \stackrel{\eqref{eq:binom_x}}{\le} m \left( \frac{4 e \delta(\alpha)}{ \lambda } \right)^{\lambda \alpha^2 m} \binom{n^2}{m}.
\end{equation}
By the choice of function $\delta$, this is at most $\beta^m \binom{n^2}{m}$.
\end{proof}

Function $\delta$ plays an important role in the last expression in \eqref{eq:P_not_regular_calc}. If, instead of function $\delta$ we wanted to use a constant $d$, we would have no choice but to assume the worst possible bound on the sets $A$ and $B$, that is $|A| = |B| = \eps n$. But then in order for \eqref{eq:P_not_regular_calc} to be at most $\beta^m \binom{n^2}{m}$, we would need to choose $d$ much smaller than $\eps$. With foresight, this would cause issues with the base of the induction in the proof of Theorem \ref{thm:KLR} where we need the opposite, that $\eps$ is significantly smaller than $d$.

\subsection{The R\"odl-Ruci\'nski deletion method}
\label{sec:H_distribution}

The following result is obtained as a straightforward application of the deletion method by R\"odl and Ruci\'nski \cite{rodl1995threshold}, thus we leave the proof for the appendix. A version of Lemma \ref{lemma:deletion_graph} is implicit in \cite{rodl1995threshold} and a related statement can be found in \cite[Proposition 3.6]{schacht2016extremal}. 


\begin{definition} \label{def:deg_H}
Given a graph $G \subseteq K_n^H$ and a pair of vertices $e = (v_a, v_b) \in V_a \times V_b$, for some distinct $a, b \in \{1, \ldots, v(H)\}$, we denote with $\deg_H(e, G)$ the number of canonical copies of $H$ in $G$ which contain $v_a$ and $v_b$.
\end{definition}

Note that in the previous definition we do not require that $(v_a, v_b)$ is an edge in $G$. If $ab$ is an edge in $H$, then, of course, $\deg_H(e, G) = 0$ for any pair $e \notin G$.

\begin{lemma} \label{lemma:deletion_graph}
    Let $H$ and $H' \subset H$ be graphs on the same vertex set, and let $\{a, b\} \in V(H)$ be two vertices such that $ab \in H$ and $ab \notin H'$. For any $\beta, \xi > 0$ there exists $T > 0$ such that the following holds for every sufficiently large $n$ and $m \ge n^{2-1/m_2(H)}$: Let $\cE$ denote the event that there exists a subset of edges $X \subseteq \Hpnm$ of size $|X| \le \xi m$ such that
    \begin{equation} \label{eq:deletion}
       \frac{1}{n^2} \sum_{e \in V_a \times V_b} \deg_{H'}(e, \Hpnm \setminus X)^2 < T \mu_e^2,
    \end{equation}
    where
    $$
        \mu_e = \frac{n^{v(H)} (m/n^2)^{e(H')}}{n^2}
    $$
    is the expected number of copies of $H'$ on $e \in V_a \times V_b$. Then $\Pr[\cE] > 1 - \beta^m$.
\end{lemma}

Even though we get a superexponentially small probability of failure, unlike in Theorem \ref{thm:KLR} here we do not require $\Hpnm$ to have any special structure. Intuitively, the previous lemma tells us that for a pair $e \in V_a \times V_b$ chosen uniformly at random, the second moment of the random variable $\deg_{H'}(e, \Hpnm \setminus X)$ is bounded by $T \mu_e^2$. This will allow us to say something about the distribution of the copies of $H'$ in the course of the proof of the K\L R conjecture:

\begin{lemma} \label{lemma:second_moment}
Let $H$ be a graph, and let $G \subseteq K_n^H$. If $G$ contains at least $\gamma n^2 \mu_e$ canonical copies of $H$ and
$$
    \frac{1}{n^2} \sum_{e \in V_a \times V_b} \deg_H(e, G)^2 < T \mu_e^2,
$$
for some $\gamma, T, \mu_e > 0$ and distinct $a, b \in V(H)$, then there are at least $\gamma n^2 / (4T)$ pairs of vertices $e \in V_a \times V_b$ such that $\deg_{H}(e, G) \ge \gamma \mu_e / 2$.
\end{lemma}
\begin{proof}
Let $Y := \deg_H(e, G)$ for $e$ chosen uniformly at random from $V_a \times V_b$. Then $\mathbb{E}[Y] \ge \gamma \mu_e$ and $\mathbb{E}[Y^2] < T \mu_e^2$, thus by the Paley-Zygmund  inequality we have
$$
    \Pr(Y \ge \mathbb{E}[Y] / 2) \ge \frac{\mathbb{E}[Y]^2}{4 \mathbb{E}[Y^2]} > \gamma / (4T).
$$
\end{proof}

It should be noted that $\mu_e$ in Lemma \ref{lemma:second_moment} is not necessarily as defined in Lemma \ref{lemma:deletion_graph}, and the statement holds for any value of $\mu_e$ as long as the conditions are satisfied.

\section{K\L R conjecture via induction}
\label{sec:proof}

Let $\cG(H, n, \eps, \delta)$ be the family of subgraphs $G \subseteq K_n^H$ such that $G(V_i, V_j)$ is $(\eps,\delta)$-lower-regular for $ij \in H$ without requirements on the actual density of $G(V_i, V_j)$, and let $\cG(H, n, m, \eps, \lambda) = \cG(H, n, \eps, \lambda) \cap \cG(H, n, m)$. Using this notation, Theorem \ref{thm:KLR} conditions on $\Hnm \in \cG(H, n, m, \eps, \lambda)$.

At a high level, hypergraph containers \cite{balogh2015independent,saxton2015hypergraph} tell us that if a graph is $H$-free, then --- and this is where the magic happens --- it is contained in a graph from a small family $\cC(H)$, and each graph in $\cC(H)$ is an `almost' $H$-free subgraph of $K_n^H$. The latter implies that each graph in $\cC(H)$ contains a large induced bipartite subgraph which is sparse. As $(\eps, \lambda)$-lower-regular graphs are uniformly dense, the existence of such a sparse subgraph severely limits the number of subgraphs of $C \in \cC(H)$ which belong to $\cG(H, n, m, \eps, \lambda)$. Summing over all graphs in $\cC(H)$ thus gives the desired upper bound on the number of $H$-free graphs in $\cG(H, n, m, \eps, \lambda)$. While this proof is very elegant and short, it relies on the existence of such a collection of \emph{containers} and it is far from obvious why this collection should exist at all.

Even though we will arrive at a similar situation in our proof, namely that we are counting the number of $(\eps, \lambda)$-lower-regular subgraphs of a sparse graph, the way we get to it is by applying the induction hypothesis of a suitable strengthening of the K\L R conjecture. In particular, we prove the following:

\begin{theorem} \label{thm:main}
Given a graph $H$, a subgraph $H' \subseteq H$ with $V(H) = V(H')$, constants $\lambda, \beta, d > 0$ and a function $\delta \colon (0, 1] \rightarrow (0, 1)$, there exist $C, \eps, \xi, \gamma > 0$ such that the following holds for every sufficiently large $n$ and $m \ge Cn^{2-1/m_2(H)}$: For any $D \in \cG(H \setminus H', n, \eps, \delta)$ with $d_D(V_i, V_j) \ge d$ for $ij \in H \setminus H'$, all but at most
$$
    \beta^m \binom{n^2}{m}^{e(H')}
$$
graphs $G \in \cG(H', n, m, \eps, \lambda)$ have the property that for every $X \subset E(G)$ of size $|X| \le \xi m$, $D \cup (G \setminus X)$ contains at least $\gamma n^{v(H)} (m/n^2)^{e(H')}$ canonical copies of $H$.
\end{theorem}

As $\Hnm$ belongs to $\cG(H', n, m, \eps, \lambda)$ with high probability, for any given $\eps, \lambda > 0$, the upper bound on the number of excluded graphs in Theorem \ref{thm:main} implies the upper bound on the corresponding conditional probability akin to the one in Theorem \ref{thm:KLR}. Therefore, taking $H = H'$ results in a robust version of Theorem \ref{thm:KLR}.

Let us motivate the role of the dense graph $D$. Let $H$ be a graph, choose an edge $h = ab \in H$ and let $H' = H \setminus h$. Suppose that Theorem \ref{thm:main} holds for $H$ and $H'$. We show that it then implies the K\L R conjecture for $H$. 

First thing to have in mind is that we will `reveal' edges of $G \in \cG(H, n, m, \eps, \lambda)$ as a sequence of graphs $G_1, \ldots, G_{2z}, G^*$, for some constant $z$, where each $G_i \in \cG(H', n, m', \eps, \lambda/2)$ for $m' = m/2z$ and $G^*$ is the remaining part between $V_a$ and $V_b$. Start with $D$ being the complete bipartite graph between $V_a$ and $V_b$ and repeat the following for $i \in \{1, \ldots, 2z\}$ sequentially: take a graph $G_i$ and remove from $D$ all the edges which together with $G_i$ span a copy of $H$. As long as $D$ is not too (locally) sparse, Theorem \ref{thm:main} and the machinery from Section \ref{sec:H_distribution} (hence the need for the set $X$) tell us that all but $\beta^{m'} \binom{n^2}{m'}^{e(H')}$ graphs $G_i$ are such that $D \cup G_i$ covers at least $\zeta n^2$ edges in $D$ with copies of $H$. Therefore, for $z = \lceil 1 / \zeta \rceil$, all but at most $(2\beta)^{zm'} \binom{n^2}{m'}^{2ze(H')}$ sequences $(G_1, \ldots, G_{2z})$ are such that in some $z$ (out of $2z$) steps we see a significant decrease in the number of edges in $D$. 

Before proceeding further, let us make an important remark: It is crucial here that we do not require that there is a significant decrease in every step, in which case we would need $z$ rounds, as that would give us only $\beta^{m/z}$. To upper bound this by $\hat \beta^m$ for a given constant $\hat \beta > 0$, we would need to impose a circular dependency: $z = \lceil 1/\zeta \rceil$, $\zeta = \zeta(\beta)$ and $\beta = \beta(z, \hat \beta)$.

Once we have a sparse graph $D$ it is very unlikely (Lemma \ref{lemma:restricted_subgraph}) that an $(\eps, \lambda)$-lower-regular graph $G^*$ with $m$ edges between $V_a$ and $V_b$ is contained within $D$. Any edge $G^* \setminus D$ closes a copy of $H$ with $G_1 \cup \ldots \cup G_{2z}$, hence we see that there are at most $(2\beta)^{m/2} \binom{n^2}{m'}^{2ze(H)}$ sequences $(G_1, \ldots, G_{2z}, G^*)$ such that $G_1 \cup \ldots \cup G_t \cup G^*$ is $H$-free. To go from the number of $H$-free sequences to the number of $H$-free graphs, it suffices to observe that any $G \in \cG(H, n, m, \eps, \lambda)$ which is $H$-free gives rise to many unique $H$-free sequences, which follows immediately from a straightforward modification of \cite[Lemma 4.1]{gerke_steger_2005}:

\begin{lemma} \label{lemma:edge_subgraph}
For all $0 < \eps < 1/6$, there exists a constant $C$ such that the following holds. Let $G \in \cG_m(H, n, m, \eps, \lambda)$, for some graph $H$, and consider the graph $G' \subseteq G$ formed by taking $m' > Cn$ edges from each $G(V_i, V_j)$, for $ij \in H$, uniformly at random. Then, with high probability $G' \in \cG(H, n, m', \eps, \lambda/2)$.
\end{lemma}

The proof of Theorem \ref{thm:main} proceeds in a similar way, by iteratively replacing a sparse pair with a dense one, eventually arriving at the point where the whole configuration is composed of dense pairs -- in which case the statement deterministically holds. Theorem \ref{thm:main} and the argument used to prove it, at least at a high level, are largely based on \cite[Theorem 2.1]{conlon2014klr} and the multiple-exposure ideas used there. These ideas, in turn, closely follow the method of Schacht \cite{schacht2016extremal} and its subsequent modification by Samotij \cite{samotij2014stability} which can be further traced to the seminal paper by R\"odl and Ruci\'nski \cite{rodl1995threshold}. The implementation is somewhat different in our case and, due to proving the claim with a superexponentially small probability of failure, we believe some parts of the argument even become easier.

\begin{proof}[Proof of Theorem \ref{thm:main}]
Consider some graph $H$ and a subgraph $H' \subseteq H$. If $H'$ is an empty graph then $D \in \cG(H, n, \eps, \delta)$ and, as $d_D(V_i, V_j) > d$ for $ij \in H$, it is well known (and easy to prove) that every such $D$ contains at least $(d\cdot \delta(\eps))^{e(H)} n^{v(H)}$ canonical copies of $H$, for sufficiently small $\eps = \eps(H, d, \delta)$. Suppose that the statement holds for all strict subgraphs $H'' \subset H'$ (having the role of $H'$). We show that then it also holds for $H'$.

We start by defining the constants. Let $\delta_0$ be a function as defined in Lemma \ref{lemma:restricted_subgraph}, and set
$$
    d_0 = \delta_0(1), \quad \beta_1 = \beta 3^{-e(H')}, \quad \beta_2 = \beta_1^2 2^{-4e(H')}.
$$
We will apply the induction hypothesis with $H''$ (as $H'$), $\lambda/2$ (as $\lambda$), $\beta_2 = \beta_1^2 2^{-4e(H')}$ (as $\beta$), $\min\{d, d_0\}$ (as $d$), and $\delta'(x) = \min\{\delta(x), \delta_0(x)\}$ (as $\delta$). Let $\eps, \xi'', \gamma'' > 0$ be constants corresponding to these parameters. We also invoke Lemma \ref{lemma:deletion_graph} with $H''$ (as $H'$), $\beta_2$ (as $\beta$) and $\xi' = \min\{\xi''/2, \lambda \eps^2 / 4\}$ (as $\xi$), and let $T$ be the corresponding constant. Finally, set
$$
    z = \lceil 4T / (\gamma'')^2 \rceil, \quad \xi = \xi' / (2z), \quad \gamma' = \gamma'' \lambda \eps^2 / 8, \quad \gamma = \gamma' / (2z)^{e(H')}.
$$

Choose an edge $h = ab \in H'$ and let $H'' = H' \setminus h$ be the graph obtained from $H'$ by removing $h$. Importantly, $H''$, $H'$ and $H$ are all on the same vertex set, even if some of them might have isolated vertice. We prove the statement for $m \ge Cn^{2-1/m_2(H)}$ divisible by $2z$, which is easily seen to imply the theorem for any other sufficiently large $m$. To this end, let 
$$
    m' = m / (2z), \quad \mu_e = \frac{n^{v(H)} (m'/n^2)^{e(H'')}}{n^2}, \quad \mu = n^{v(H)} (m'/n^2)^{e(H')} = \mu_e \cdot m',
$$
and note that
$$
    \gamma n^{v(H)} (n/m^2)^{e(H')} = \gamma' \mu.
$$
Observe that $\mu$ is the expected number of copies of $H'$ in $\Hpnmm$, and $\mu_e$ is the expected number of copies of $H''$ in $\Hppnmm$ sitting on a particular pair of vertices in $V_a \times V_b$. Consequently, these values upper bound the expected number of such copies of $H$ in $\Hpnmm \cup D$, and $H \setminus h$ in $\Hppnmm \cup D$.

Given a graph $G \in \cG(H', n, m', \eps, \lambda/2)$, define the set $\cP(G)$ of `poor' pairs of vertices as follows (see Definition \ref{def:deg_H}):
$$
    \cP(G) = \left\{ e \in V_a \times V_b \colon \deg_{H \setminus h}(e, D \cup G^-) < \gamma'' \mu_e / 2 \right\},
$$
where $G^-$ is the graph obtained from $G$ by removing all the edges between $V_a$ and $V_b$. Note that we could have defined $\cP(G)$ with $G = G^-$ instead, however we want to emphasise that we  ignore the edges of $G$ between $V_a$ and $V_b$.


The following definition encompasses the `advancing' argument used in the proof.

\textbf{Advancing property $\cA(\hat G)$.} Given a graph $\hat G \subseteq K_n^{H'}$, we say that a graph $G \in \cG(H', n, m', \eps, \lambda/2)$ is \emph{advancing} with respect to $\hat G$ if, for every $X \subseteq E(G)$ of size $|X| \le \xi' m'$, at least one of the following holds:
\begin{enumerate}[(A)]
    \item \label{prop:A} $D \cup \hat G \cup (G \setminus X)$ contains at least $\gamma' \mu$ canonical copies of $H$;
    \item \label{prop:B} $|\cP(\hat G \cup (G \setminus X ))| < |\cP(\hat G)| - n^2 / z$.
\end{enumerate}

\begin{claim} \label{claim:main_counting_claim}
For any $\hat G \subseteq K_n^{H'}$, all but at most $\beta_2^{m'}\binom{n^2}{m'}^{e(H')}$
graphs $G \in \cG(H', n, m', \eps, \lambda/2)$ have the property $\cA(\hat G)$.
\end{claim}
We postpone the proof of the claim until the end.

We now come to the key concept in our proof. Given a subset $S \subseteq \{1, \ldots, 2z\}$, a sequence of graphs $\mathbf{G} = (G_1, \ldots, G_{2z})$, and a sequence of subsets $\mathbf{X} = (X_{i})_{i \in S}$, where for each $i \in S$ we have $X_{i} \subseteq E(G_{i})$ and $|X_{i}| \le \xi' m'$, we say that $\mathbf{G}$ is \emph{$(S, \mathbf{X})$-bad} if, for each $i \in [2z] \setminus S$, $G_i \notin \cA(\hat G_{i-1})$ where
\begin{equation} \label{eq:G_hat}
    \hat G_i = \bigcup_{\substack{j \in S\\j\le i}} G_j \setminus X_j.
\end{equation}
We say that $\mathbf{G}$ is \emph{bad} if it is $(S, \mathbf{X})$-bad for some choice of $S$ and $\mathbf{X}$ with $|S| \le z$.

Let us demonstrate the usefulness of this definition. Consider some $G \in \cG(H', n, m, \eps, \lambda)$. By Lemma \ref{lemma:edge_subgraph}, we can partition $G$ into edge-disjoint graphs $G_1, \ldots, G_{2z} \in \cG(H', n, m', \eps, \lambda/2)$. We claim that if $\mathbf{G} = (G_1, \ldots, G_{2z})$ is not bad, then $D \cup (G \setminus X)$ contains at least $\gamma \mu$ canonical copies of $H$, for any $X \subseteq E(G)$ of size $|X| \le \xi m$. Consider some $X \subseteq E(G)$ of size $|X| \le \xi m = \xi' m'$. Set $S = \emptyset$ and, sequentially, for each $i \in \{1, \ldots, 2z\}$ do the following: if $G_i \in \cA(\hat G_{i-1})$, where $\hat G_i$ is as given by \eqref{eq:G_hat} with respect to the current set $S$, add $i$ to $S$ and set $X_i = X \cap E(G_i)$. As we have assumed that $\mathbf{G}$ is not bad, we have $|S| > z$. Consider only the first $z+1$ elements of $S$, say $s_1, \ldots, s_{z+1}$. If $D \cup \hat G_{s_{z+1}}$ contains less than $\gamma \mu$ copies of $H$ then, by \ref{prop:B}, we have
$$
    |\cP(\hat G_{s_{i+1}})| < |\cP(\hat G_{s_i})| - n^2 / z
$$
for every $i \in \{1, \ldots, z\}$. In particular, $|\cP(\hat G_{s_{z+1}})| < 0$, which cannot be. Therefore, $D \cup \hat G_{s_{z+1}}$ and, consequently, $D \cup (G \setminus X)$, contain the desired number of canonical copies of $H$.

We now count the number of bad sequences with each $G_i \in \cG(H', n, m', \eps, \lambda/2)$. First, we can choose $S \subseteq \{1, \ldots, 2z\}$ of size $|S| \le z$ in at most $2^{2z}$ ways. For a fixed choice of $S$, say $S = \{s_1, \ldots, s_{z'}\}$ for some $z' \le z$, there are at most $\binom{n^2}{m'}^{e(H') z'}$ ways to choose $G_{s_1}, \ldots, G_{s_{z'}}$. For each such choice, we can choose $\mathbf{X}$ in at most $2^{m e(H')}$ ways. For each $i \notin S$, by Claim \ref{claim:main_counting_claim}, there are at most $\beta_2^{m'} \binom{n^2}{m'}^{e(H')}$ choices for $G_i \notin \cA(\hat G_{i-1})$. All together, this gives at most
\begin{equation}
    \label{eq:bad_seq_count}
    2^{2z} \cdot 2^{m e(H')}  \cdot \binom{n^2}{m'}^{e(H') z'} \cdot \left( \beta_2^{m'} \binom{n^2}{m'}^{e(H')} \right)^{2z - z'} <
    2^{2m e(H')} \beta_2^{m/2} \cdot \binom{n^2}{m'}^{e(H') \cdot 2z} = \beta_1^m \binom{n^2}{m'}^{e(H') \cdot 2z}
\end{equation}
such bad sequences.

Finally, we use \eqref{eq:bad_seq_count} to upper bound the number of graphs $G \in \cG(H', n, m, \eps, \lambda)$ which do not satisfy the property of the theorem. Note that Lemma \ref{lemma:edge_subgraph} implies that there are at least
\begin{equation} \label{eq:G_into_seq}
    0.99 \cdot \left( \frac{m!}{(m'!)^{2z}} \right)^{e(H')}
\end{equation}
different ways to partition $G$ into a sequence $G_1, \ldots, G_{2z} \in \cG(H', n, m', \eps, \lambda/2)$. If $G$ does not satisfy the property of Theorem \ref{thm:main} then each of these sequences is bad, and two different graphs $G$ and $G'$ give different families of sequences. Therefore, combining \eqref{eq:bad_seq_count} and \eqref{eq:G_into_seq} we get the desired upper bound on the number of `bad' graphs $G \in \cG(H', n, m, \eps, \lambda)$:
$$
    \frac{\beta_1^m \binom{n^2}{m'}^{e(H') \cdot 2z}}{0.99 \cdot \left( \frac{m!}{(m'!)^{2z}} \right)^{e(H')}} < 2 \beta_1^m \left( \frac{(n^2)^{m}}{m!} \right)^{e(H')} 
    \le \beta_1^m 2^{m e(H') + 1} \binom{n^2}{m}^{e(H')} < \beta^m \binom{n^2}{m}^{e(H')}.
$$
In the penultimate inequality we assumed $m < n^2/2$. If this is not the case, then we are well within the dense regime and the theorem is both known and easy to prove.

It remains to prove Claim \ref{claim:main_counting_claim}.

\begin{proof}[Proof of Claim \ref{claim:main_counting_claim}]
Let $Q$ stand for the bipartite graph between $V_a$ and $V_b$, with the edges being pairs from $\cP(\hat G)$. Following the proof of \cite[Claim 2.6]{conlon2014klr}, we distinguish two cases. 

\textbf{Case 1. $e(Q) \le d_0 n^2$ or $Q$ is not $(\eps, \delta_0)$-lower-regular.} 

Let $G \in \cG(H', n, m', \eps, \lambda/2)$ be the graph such that $|G(V_a, V_b) \setminus Q| \ge \lambda \eps^2 m' / 2$. By Lemma \ref{lemma:restricted_subgraph}, all but at most $\beta_2^{m'} \binom{n^2}{m'}^{e(H')}$ graphs have such a property. For any $X \subseteq G$ of size $|X| \le \xi' m'$, we have $|G(V_a, V_b) \setminus (Q \cup X)| \ge \lambda \eps^2 m' / 4$ and, as each pair of vertices corresponding to edges $G(V_a, V_b) \setminus Q$ is contained in at least $\gamma'' \mu_e / 2$ copies of $H \setminus h$ in $D \cup \hat G^-$, this gives
$$
    (\gamma'' \mu_e / 2) \cdot (\lambda \eps^2 m' / 4) = \gamma' \mu
$$
copies of $H$ in $D \cup \hat G \cup (G \setminus X)$. Therefore, the part \ref{prop:A} of the property $\cA(\hat G)$ holds.

\textbf{Case 2. $e(Q) > d_0 n^2$ and $Q$ is $(\eps, \delta_0)$-lower-regular.} 

Let $D' \in \cG(H \setminus H'', n, \eps, \delta')$ be a graph obtained from $D$ by adding all the edges from $Q$. By the induction hypothesis and Lemma \ref{lemma:deletion_graph} (with $H''$ as $H$), all but at most $\beta_2^{m'} \binom{n^2}{m'}^{e(H'')}$ graphs $G' \in \cG(H'', n, m', \eps, \lambda/2)$ have the following two properties:
\begin{enumerate}[(i)]
    \item \label{prop:robust_copies} for every $X \subseteq E(G')$ of size $|X| < 2 \xi' m'$, $D' \cup (G' \setminus X)$ contains at least $\gamma'' \mu_e n^2$ canonical copies of $H$;
    \item \label{prop:bound_second_moment} there exists $X' \subseteq E(G')$ of size $|X'| \le \xi' m'$ such that
    \begin{equation} \label{eq:bound_second_moment}
        \frac{1}{n^2} \sum_{e \in V_a \times V_b} \deg_{H''}(e, D' \cup (G' \setminus X'))^2 < T \mu_e^2.
    \end{equation}
\end{enumerate}
Let $G' \in \cG(H'', n, m', \eps, \lambda/2)$ be a graph which satisfies these two properties, and let $X' \subseteq E(G')$ be a set of size $|X'| \le \xi' m'$ as given by \ref{prop:bound_second_moment}. Then, for any $X \subseteq E(G')$ of size $|X| \le \xi m = \xi' m'$, Lemma \ref{lemma:second_moment} with $D' \cup (G' \setminus (X \cup X'))$ (as $G$), \ref{prop:robust_copies} and \eqref{eq:bound_second_moment},  tells us that at least $n^2 / z$ edges in $Q$ belong to at least $\gamma'' \mu_e / 2$ copies of $H$. Regardless of how we extend $G'$ to a graph $G \in \cG(H', n, m', \eps, \lambda/2)$, none of these edges belong to $\cP(\hat G \cup (G \setminus X))$ hence the part \ref{prop:B} holds.
\end{proof}

This finishes the proof of Theorem \ref{thm:main}.
\end{proof}

It is worth noting that the fact we need to remove a subset of edges such that \eqref{eq:bound_second_moment} holds is the sole reason why we need the set $X$ in Theorem \ref{thm:main}.

\section{Concluding remarks}
\label{sec:concluding}

The approach used to prove the K\L R conjecture can also be applied in other contexts. As an example, in a forthcoming companion note \cite{nenadov2021szemeredisparse} to this paper we give a short proof of a theorem of Balogh, Morris, and Samotij \cite{balogh2015independent} which states that a randomly chosen $m$-element subset of $[n]$ fails to contain a $k$-term arithmetic progression with probability at most $\beta^m$ for any constant $\beta > 0$ provided $m \ge Cn^{1 - 1/(k-1)}$ where $C = C(\beta)$ is sufficiently large.

In general, the proof strategy requires three ingredients: (i) the deterministic statement holds for structures of arbitrarily small density; (ii) a given structure can be partitioned in many different ways into smaller substructures which are captured by the inductive statement; (iii) an upper bound on the second-moment akin to the one in \eqref{eq:deletion} can be achieved by potentially removing a small subset of elements. While it is possible to wrap up these conditions in a general theorem, it is not clear that such a theorem would match the wide spectrum of applications of the hypergraph containers or other related  transference results \cite{conlon2014extremal,schacht2016extremal}. Thus the value lies in the method itself and, if nothing else, its applications to specific examples which uncover transparent and intuitive proofs.

\textbf{Acknowledgment.} The author would like to thank Angelika Steger for many stimulating discussions about the K\L R conjecture, and to Wojciech Samotij for helpful comments on the manuscript.

{\small \bibliographystyle{abbrv} \bibliography{references}}

\begin{thebibliography}{10}

\bibitem{balogh2015independent}
J.~Balogh, R.~Morris, and W.~Samotij.
\newblock Independent sets in hypergraphs.
\newblock {\em Journal of the American Mathematical Society}, 28(3):669--709,
  2015.

\bibitem{conlon2014extremal}
D.~Conlon, J.~Fox, and Y.~Zhao.
\newblock Extremal results in sparse pseudorandom graphs.
\newblock {\em Advances in Mathematics}, 256:206--290, 2014.

\bibitem{conlon2016combinatorial}
D.~Conlon and W.~T. Gowers.
\newblock Combinatorial theorems in sparse random sets.
\newblock {\em Annals of Mathematics}, pages 367--454, 2016.

\bibitem{conlon2014klr}
D.~Conlon, W.~T. Gowers, W.~Samotij, and M.~Schacht.
\newblock On the {K{\L}R} conjecture in random graphs.
\newblock {\em Israel Journal of Mathematics}, 203(1):535--580, 2014.

\bibitem{frankl1986large}
P.~Frankl and V.~R{\"o}dl.
\newblock Large triangle-free subgraphs in graphs without {$K_4$}.
\newblock {\em Graphs and Combinatorics}, 2(1):135--144, 1986.

\bibitem{furedi1994random}
Z.~F{\"u}redi.
\newblock Random ramsey graphs for the four-cycle.
\newblock {\em Discrete Mathematics}, 126(1-3):407--410, 1994.

\bibitem{gerke_habil}
S.~Gerke.
\newblock {\em Random graphs with constraints}.
\newblock Habilitationsschrift, Institut f\"ur Informatik, Technische
  Universit\"at M\"unchen, 2005.

\bibitem{gerke2007small}
S.~Gerke, Y.~Kohayakawa, V.~R{\"o}dl, and A.~Steger.
\newblock Small subsets inherit sparse $\varepsilon$-regularity.
\newblock {\em Journal of Combinatorial Theory, Series B}, 97(1):34--56, 2007.

\bibitem{gerke2007probabilistic}
S.~Gerke, M.~Marciniszyn, and A.~Steger.
\newblock A probabilistic counting lemma for complete graphs.
\newblock {\em Random Structures \& Algorithms}, 31(4):517--534, 2007.

\bibitem{gerke2007k}
S.~Gerke, H.~J. Pr{\"o}mel, T.~Schickinger, A.~Steger, and A.~Taraz.
\newblock {$K_4$}-free subgraphs of random graphs revisited.
\newblock {\em Combinatorica}, 27(3):329--365, 2007.

\bibitem{gerke2004k5}
S.~Gerke, T.~Schickinger, and A.~Steger.
\newblock {$K_5$}-free subgraphs of random graphs.
\newblock {\em Random Structures \& Algorithms}, 24(2):194--232, 2004.

\bibitem{gerke_steger_2005}
S.~Gerke and A.~Steger.
\newblock {\em The sparse regularity lemma and its applications}, page
  227–258.
\newblock London Mathematical Society Lecture Note Series. Cambridge University
  Press, 2005.

\bibitem{haxell1995turan}
P.~E. Haxell, Y.~Kohayakawa, and T.~{\L}uczak.
\newblock Tur{\'a}n's extremal problem in random graphs: Forbidding even
  cycles.
\newblock {\em Journal of Combinatorial Theory, Series B}, 64(2):273--287,
  1995.

\bibitem{haxell1996turan}
P.~E. Haxell, Y.~Kohayakawa, and T.~{\L}uczak.
\newblock Tur{\'a}n's extremal problem in random graphs: Forbidding odd cycles.
\newblock {\em Combinatorica}, 16(1):107--122, 1996.

\bibitem{janson1990}
S.~{Janson}, T.~{{\L}uczak}, and A.~{Ruci\'nski}.
\newblock {An exponential bound for the probability of nonexistence of a
  specified subgraph in a random graph}.
\newblock {Random graphs '87, Proc. 3rd Int. Semin., Pozna\'n/Poland 1987,
  73-87 (1990).}, 1990.

\bibitem{janson2011random}
S.~Janson, T.~{\L}uczak, and A.~Ruci\'nski.
\newblock {\em Random graphs}.
\newblock John Wiley \& Sons, 2011.

\bibitem{kohayakawa2018anti_ramsey}
Y.~{Kohayakawa}, P.~B. {Konstadinidis}, and G.~O. {Mota}.
\newblock {On an anti-{R}amsey threshold for sparse graphs with one triangle}.
\newblock {\em {J. Graph Theory}}, 87(2):176--187, 2018.

\bibitem{kohayakawa1997threshold}
Y.~Kohayakawa and B.~Kreuter.
\newblock Threshold functions for asymmetric {R}amsey properties involving
  cycles.
\newblock {\em Random Structures \& Algorithms}, 11(3):245--276, 1997.

\bibitem{kohayakawa1996arithmetic}
Y.~Kohayakawa, T.~{\L}uczak, and V.~R{\"o}dl.
\newblock Arithmetic progressions of length three in subsets of a random set.
\newblock {\em Acta Arithmetica}, 75:133--163, 1996.

\bibitem{kohayakawa1997k}
Y.~Kohayakawa, T.~{\L}uczak, and V.~R{\"o}dl.
\newblock On {$K_4$}-free subgraphs of random graphs.
\newblock {\em Combinatorica}, 17(2):173--213, 1997.

\bibitem{kohayakawa2004turan}
Y.~Kohayakawa, V.~R{\"o}dl, and M.~Schacht.
\newblock The {T}ur{\'a}n theorem for random graphs.
\newblock {\em Combinatorics, Probability and Computing}, 13(1):61--91, 2004.

\bibitem{kohayakawa2014upper}
Y.~Kohayakawa, M.~Schacht, and R.~Sp{\"o}hel.
\newblock Upper bounds on probability thresholds for asymmetric ramsey
  properties.
\newblock {\em Random Structures \& Algorithms}, 44(1):1--28, 2014.

\bibitem{luczak2000Hexact}
T.~{{\L}uczak}.
\newblock {On triangle-free random graphs}.
\newblock {\em {Random Struct. Algorithms}}, 16(3):260--276, 2000.

\bibitem{nenadov2021szemeredisparse}
R.~Nenadov.
\newblock Small subsets without $k$-term arithmetic progressions.
\newblock 2021.
\newblock Preprint.

\bibitem{steger1996countingHfree}
H.~J. {Pr\"omel} and A.~{Steger}.
\newblock {Counting \(H\)-free graphs}.
\newblock {\em {Discrete Math.}}, 154(1-3):311--315, 1996.

\bibitem{rodl1995threshold}
V.~R{\"o}dl and A.~Ruci{\'n}ski.
\newblock Threshold functions for {R}amsey properties.
\newblock {\em Journal of the American Mathematical Society}, 8(4):917--942,
  1995.

\bibitem{samotij2014stability}
W.~Samotij.
\newblock Stability results for random discrete structures.
\newblock {\em Random Structures \& Algorithms}, 44(3):269--289, 2014.

\bibitem{saxton2015hypergraph}
D.~Saxton and A.~Thomason.
\newblock Hypergraph containers.
\newblock {\em Inventiones mathematicae}, 201(3):925--992, 2015.

\bibitem{schacht2016extremal}
M.~Schacht.
\newblock Extremal results for random discrete structures.
\newblock {\em Annals of Mathematics}, pages 333--365, 2016.

\bibitem{szabo2003turan}
T.~Szab{\'o} and V.~H. Vu.
\newblock Tur{\'a}n's theorem in sparse random graphs.
\newblock {\em Random Structures \& Algorithms}, 23(3):225--234, 2003.

\end{thebibliography}

\newpage
\appendix 
\section{Proof of Lemma \ref{lemma:deletion_graph}}
\label{sec:rodl_rucinski_deletion_appendix}

The following lemma is an analogue of \cite[Lemma 4]{rodl1995threshold}.

\begin{lemma} \label{lemma:deletion_general}
    Let $Y$ be a set with $N$ elements, and $\cS$ be a family of $s$-element subsets of $Y$, for some integers $s$ and $N$. Let $k < N/4$ be an integer and $T > 1$. Then, for all but at most 
    $$
        T^{-k/s} \binom{N}{m}
    $$
    subsets $Y_m \subseteq Y$ of size $m$, there exists a subset $X \subseteq Y_m$ of size $|X| \le k$ such that $Y_m \setminus X$ contains at most
    $$
        2^s T |\cS| (m/N)^{s}
    $$
    sets from $\cS$.
\end{lemma}
\begin{proof}
Consider some $Y_m \subseteq Y$ of size $|Y_m| = m$. Following the proof of \cite[Lemma 4]{rodl1995threshold}, let $Z$ be the number of $t$-element sequences $\mathbf{S} = (S_1, \ldots, S_t)$ of disjoint sets from $\cS$ which are contained in $Y_m$, where $t = \lceil k/s \rceil$. If $Y_m$ does not contain a desired subset $X$, then no matter how we choose the first $t' < t$ elements of $\mathbf{S}$, there are still at least $2^s T |\cS| (m/N)^s$ choices for $S_{t' + 1}$. Overall, this implies
$$
    Z \ge \Big( 2^s T |\cS| (m/N)^s \Big)^t.
$$

Suppose that we choose $Y_m$ uniformly at random among all $m$-element subsets of $Y$, and let $\cE$ be the event that a desired subset $X \subseteq Y_m$ does not exist. Then a particular sequence $\mathbf{S}$ appears in $Y_m$ with probability
$$
    \binom{N - ts}{m - ts} / \binom{N}{m} < \left( 2m / N \right)^{ts},
$$
with room to spare, thus $\mathbb{E}[Z] < \left( 2^s |\cS| (m/N)^s \right)^{t}$. By Markov's inequality, we have
$$
    \Pr[\cE] \le \Pr\left[Z \ge \Big( 2^s T |\cS| (m/N)^s \Big)^t\right] \le \frac{\mathbb{E}[Z]}{\Big( 2^s T |\cS| (m/N)^s \Big)^t} \le T^{-k/s},
$$
which implies the desired bound.
\end{proof}

\begin{proof}[Proof of Lemma \ref{lemma:deletion_graph}]
    Consider some graph $\hat H$ which is a union of two copies of $H'$, say $\hat H = H_1 \cup H_2$, such that $a$ and $b$ belong to $J = V(H_1) \cap V(H_2)$. Moreover, we require that $H_1 \cap H_2$ is an induced subgraph of $\hat H$, that is $H_1 \cap H_2 = \hat H[J]$. We identify the vertex set $J$ with the subgraph it induces. There are at most $2^{v(H)}$ such graphs $\hat H$. Note that
    $$
        \sum_{e \in V_a \times V_b} \deg_{H'}(e, G)^2
    $$
    counts the number of canonical copies of all such $\hat H$ in $G$. Here we say that $\hat H$ is canonical if the vertices corresponding to $H_1$ and $H_2$ are in the correct sets.
    
    To prove the lemma, it suffices to show that for each such graph $\hat H$, for all but at most
    $$
        (\beta/2^{v(H)})^m \binom{n^2}{m}^{e(H)}
    $$
    graphs $G \in \cG(H', n, m)$ there exists $X \subseteq E(G)$ of size $|X| \le \xi 2^{-v(H)} m$ such that $G \setminus X$ contains at most
    \begin{equation} \label{eq:T_hatH}
        T 2^{-v(H)} n^{v(\hat H)} (m/n^2)^{e(\hat H)} = T 2^{-v(H)} n^{2v(H) - v(J)} (m/n^2)^{2e(H) - e(J)}
    \end{equation}
    canonical copies of $\hat H$. Indeed, as
    \begin{equation} \label{eq:J_bound_n2}
        n^{v(J)} \left( \frac{m}{n^2} \right)^{e(J)} \ge n^2
    \end{equation}
    for $m \ge n^{2 - 1/m_2(H)}$, we can upper bound \eqref{eq:T_hatH} as
    $$
        T 2^{-v(H)} n^{2v(H) - 2} (m/n^2)^{2e(H)}.
    $$
    Summing over all graphs $\hat H$ implies the lemma. The inequality \eqref{eq:J_bound_n2} follows from the fact that $J$ is missing an edge $ab$ which is present in $H$, which implies $m_2(H) \ge e(J) / (v(J) - 2)$.
    
    Consider one such $\hat H$, and apply Lemma \ref{lemma:deletion_general} with $Y$ being the edge set of $K_n^H$, $\cS$ being he family of all canonical copies of $\hat H$ in $K_n^H$, and $k = \xi 2^{-v(H)} m$. For $T' = T2^{-v(H) - e(\hat H)}$, where $T = T(\xi, \hat H)$ is sufficiently large, this gives that all but at most
    $$
        (T')^{-\xi 2^{-v(H)} m / e(\hat H)} \binom{e(H) n^2}{e(H) m} < (\beta 2^{-v(H)})^m \binom{n^2}{m}^{e(H)}
    $$
    graphs $G \subseteq K_n^{H}$ with exactly $e(H) m$ edges have the property that $G \setminus X$ contains at most
    $$
        T' n^{v(\hat H)} (2m/n^2)^{e(\hat H)} = T 2^{-v(H)} n^{v(\hat H)} (m/n^2)^{e(\hat H)}
    $$
    copies of $\hat H$, for some $X \subseteq E(G)$ of size at most $\xi 2^{-v(H)} m$. This is trivially an upper bound on such bad graphs in $\cG(H, n, m)$, which finishes the proof.
\end{proof}

\end{document}